\documentclass[12pt,reqno]{amsart}
\pdfoutput=1
\usepackage{amsmath,latexsym}
\usepackage[psamsfonts]{amssymb}
\usepackage{mathrsfs}
\usepackage{amsthm}
\usepackage[OT2,OT1]{fontenc}
\usepackage[all]{xy}
\usepackage{url}
\numberwithin{equation}{section}
\date{}
\begin{document}

\newtheorem{theorem}{Theorem}[section]
\newtheorem{corollary}[theorem]{Corollary}
\newtheorem{conjecture}[theorem]{Conjecture}
\newtheorem{lemma}[theorem]{Lemma}
\newtheorem{fact}[theorem]{Fact}
\newtheorem{proposition}[theorem]{Proposition}

\theoremstyle{remark}
\newtheorem{remark}[theorem]{Remark}

\theoremstyle{remark}
\newtheorem{example}[theorem]{Example}

\theoremstyle{remark}
\newtheorem*{ack}{\rm Acknowledgements}

\author[J\'ulius Korba\v s]{J\'ulius Korba\v s}

\thanks{ Part of this research was carried out while the author was a
member of the research teams 1/0330/13 and 2/0029/13 supported in part by the grant agency
VEGA (Slovakia).}

\title[On the characteristic rank and cup-length]
{The characteristic rank and cup-length in oriented Grassmann manifolds}

\address{Faculty of Mathematics, Physics, and Informatics\\
Comenius University\\ Mlynsk\'a dolina\\ SK-842 48 Bratislava\\ Slovakia and\\
Mathematical Institute of SAS \\ \v Stef\'anikova 49\\ SK-841 73 Bratislava \\ Slovakia }
\email{korbas@fmph.uniba.sk}

\subjclass[2000]{primary 57R20, secondary 55R25}

\keywords{Stiefel-Whitney class, characteristic rank, Grassmann manifold.}

\begin{abstract}
In the first part, this paper studies the characteristic rank of the canonical oriented
$k$-plane bundle over the Grassmann manifold $\widetilde G_{n,k}$ of oriented $k$-planes in
Euclidean $n$-space. It presents infinitely many new exact values if $k=3$ or $k=4$, as well
as new lower bounds for the number in question if $k\geq 5$. In the second part, these
results enable us to improve on the general upper bounds for the $\mathbb Z_2$-cup-length of
$\widetilde G_{n,k}$. In particular, for $\widetilde G_{2^t,3}$ $(t\geq 3)$ we prove that the
cup-length is equal to $2^t-3$, which verifies the corresponding claim of Tomohiro Fukaya's
conjecture from 2008.

\end{abstract}

\maketitle

\section{Introduction and some preliminaries}\label{sec:introduction}
Given a real vector bundle $\alpha$ over a path-connected $CW$-complex $X$, the {\em
characteristic rank of} $\alpha$, denoted $\mathrm{charrank}(\alpha)$, is defined to be
(\cite{naolekarthakur}) the greatest integer $q$, $0 \leq q \leq \mathrm{dim}(X)$, such that
every cohomology class in $H^j(X)$, $0 \leq j \leq q$, is a polynomial in the Stiefel-Whitney
classes $w_i(\alpha)\in H^i(X)$. Here and elsewhere in this paper, we write $H^i(X)$ instead
of $H^i(X;\mathbb Z_2)$.

In particular, if $TM$ is the tangent bundle of a smooth closed connected manifold $M$, then
$\mathrm{charrank}(TM)$ is nothing but the {\em characteristic rank of} $M$, denoted
$\mathrm{charrank}(M)$; this homotopy invariant of smooth closed connected manifolds was
introduced, and in some cases also computed, in \cite{korbas2010}. Results on the
characteristic rank of vector bundles over the Stiefel manifolds can be found in
\cite{knt2012}. The characteristic rank is useful, for instance, in studying the cup-length
of a given space (see \cite{korbas2010}, \cite{naolekarthakur}, and also Section
\ref{sec:newresultsoncup-length} of the present paper).

It is readily seen that the characteristic rank of the canonical $k$-plane bundle
$\gamma_{n,k}$ (briefly $\gamma$) over the Grassmann manifold $G_{n,k}$ $(k\leq n-k)$ of all
$k$-dimensional vector subspaces in $\mathbb R^n$ is equal to $\mathrm{dim}(G_{n,k})=k(n-k)$.
Indeed, as is well known (\cite{2:bor}), for the $\mathbb Z_2$-cohomology algebra
$H^\ast(G_{n,k})$ we can write
\begin{equation}\label{cohomologyalgebraofG_{n,k}}
H^\ast(G_{n,k}) = \mathbb Z_2[w_1,\dots, w_k]/I_{n,k}, \end{equation} where
$\mathrm{dim}(w_i)=i$ and the ideal $I_{n,k}$ is generated by the $k$ homogeneous components
of $(1+w_1+\dots+w_k)^{-1}$ in dimensions $n-k+1,\dots, n$; here the indeterminate $w_i$ is a
representative of the $i$th Stiefel-Whitney class $w_i(\gamma)$ in the quotient algebra
$H^*(G_{n,k})$. For the latter class $w_i(\gamma)$, we shall also use $w_i$ as an
abbreviation.

In contrast to the situation for $G_{n,k}$, the $\mathbb Z_2$-cohomology algebra
$H^\ast(\widetilde G_{n,k})$ $(k\leq n-k)$ of the \lq\lq oriented\rq\rq\, Grassmann manifold
$\widetilde G_{n,k}$ of all {\it oriented} $k$-dimensional vector subspaces in $\mathbb R^n$
is in general unknown. Since $\widetilde G_{n,1}$ can be identified with the
$(n-1)$-dimensional sphere, and the complex quadrics $\widetilde G_{n,2}$ are also well
understood special cases, we shall suppose that $k\geq 3$ throughout the paper.

In Section \ref{sec:mainresult}, we derive infinitely many new exact values if $k=3$ or
$k=4$, as well as new lower bounds for the characteristic rank of the canonical oriented
$k$-plane bundle $\widetilde \gamma_{n,k}$ (briefly $\widetilde \gamma$) over $\widetilde
G_{n,k}$ if $k\geq 5$. As a consequence, for odd $n$, we also obtain better bounds (as
compared to those known from \cite[p. 73]{korbas2010}) on the invariant
$\mathrm{charrank}(\widetilde G_{n,k})$. Then, in Section \ref{sec:newresultsoncup-length},
our results on the characteristic rank of $\widetilde \gamma$ enable us to improve on the
general upper bounds for the $\mathbb Z_2$-cup-length of $\widetilde G_{n,k}$. In particular,
for $\widetilde G_{2^t,3}$ $(t\geq 3)$ we prove that the cup-length is equal to $2^t-3$; this
verifies the corresponding claim of Fukaya's conjecture \cite[Conjecture 1.2]{fukaya}.

\section{On the characteristic rank of the canonical vector bundle over $\widetilde G_{n,k}$}
\label{sec:mainresult}

Using the notation introduced in Section \ref{sec:introduction}, we now state our main
result.

\begin{theorem} \label{th:charrankoforientedcanonicalbundles} For the canonical
$k$-plane bundle $\widetilde \gamma_{n,k}$ over the oriented Grassmann manifold $\widetilde
G_{n,k}$ $(3\leq k\leq n-k)$, with $2^{t-1}<n\leq 2^t$, we have

\begin{enumerate}

\item $\mathrm{charrank}(\widetilde \gamma_{n,3})
\left\{\begin{array}{ll}=n-2&\text{ if } n = 2^t,\\
= n-5+i&\text{ if } n=2^t -i,\,i\in \{1,\,2,\,3\},\\
\geq n-2 &\text{ otherwise; }
\end{array}\right.$

\item $\mathrm{charrank}(\widetilde \gamma_{n,4})
\left\{\begin{array}{ll}
= n-5+i&\text{ if } n = 2^t-i,\,i\in \{0,\,1,\,2,\,3\},\\
\geq n-3 &\text{ otherwise; } \\
\end{array}\right.$

\item if $k\geq 5$, then $\mathrm{charrank}(\widetilde \gamma_{n,k}) \geq n-k+1$.
\end{enumerate}
In addition, if $n$ is odd, then the replacement of the canonical bundle $\widetilde
\gamma_{n,j}$ by the corresponding manifold $\widetilde G_{n,j}$, in $(1)-(3)$, gives the
corresponding result on $\mathrm{charrank}(\widetilde G_{n,j})$.

\end{theorem}

We shall pass to a proof of this theorem after some preparations.

For the universal $2$-fold covering $p: \widetilde G_{n,k} \rightarrow G_{n,k}$ $(k\geq 3)$,
the pullback $p^*(\gamma)$ is $\widetilde \gamma$, and for the induced homomorphism in
cohomology we have that $p^*(w_i)=\widetilde w_i$ for all $i$, where $\widetilde w_i$ is an
abbreviated notation, used throughout the paper, for the Stiefel-Whitney class
$w_i(\widetilde \gamma_{n,k})$. Of course, now $\mathrm{charrank}(\widetilde \gamma_{n,k})$
is, in other words, the greatest integer $q$, $0 \leq q \leq k(n-k)$, such that $p^*:
H^{j}(G_{n,k}) \longrightarrow H^{j}(\widetilde G_{n,k})$ is surjective for all $j$, $0 \leq
j \leq q$.

To the covering $p$ there is associated a uniquely determined non-trivial line bundle $\xi$
such that $w_1(\xi) = w_1(\gamma_{n,k})$. This yields (\cite[Corollary 12.3]{milnorstasheff})
an exact sequence of Gysin type,

\begin{equation}
\longrightarrow H^{j-1}(G_{n,k})\overset{w_1}{\longrightarrow} H^{j}(G_{n,k})
\overset{p^*}{\longrightarrow} H^{j}(\widetilde G_{n,k})\longrightarrow
H^{j}(G_{n,k})\overset{w_1}{\longrightarrow} \label{Gysinsequence}
\end{equation}
As is certainly clear from the context, we write here and elsewhere
$H^{j-1}(G_{n,k})\overset{w_1}{\longrightarrow} H^{j}(G_{n,k})$ for the homomorphism given by
the cup-product with the Stiefel-Whitney class $w_1$.

Thus $p^*: H^{j}(G_{n,k}) \longrightarrow H^{j}(\widetilde G_{n,k})$ is surjective if and
only if the subgroup
\begin{equation}
\mathrm{Ker}(H^{j}(G_{n,k})\overset{w_1}{\longrightarrow} H^{j+1}(G_{n,k}))
\label{kernelofmultipbyw_1}
\end{equation}
vanishes.

By \eqref{cohomologyalgebraofG_{n,k}}, a $\mathbb Z_2$-polynomial
\begin{equation}
p_j(w_1,\dots, w_k) = \sum_{i_1+2i_2+\cdots +ki_k=j} a_{i_1, i_2,\dots,
i_k}w_1^{i_1}w_2^{i_2}\cdots w_k^{i_k}, \end{equation} with at least one coefficient $a_{i_1,
i_2,\dots, i_k}\in \mathbb Z_2$ nonzero, represents zero in $H^j(G_{n,k})$ precisely when
there exist some polynomials $q_i(w_1,\dots, w_k)$ (briefly $q_i$) such that
$$p_j=q_{j-n+k-1}{\bar w}_{n-k+1}+ \cdots + q_{j-n} {\bar w}_{n},$$ where
${\bar w}_i(w_1,\dots, w_k)$ (briefly ${\bar w}_i$) is the homogeneous component of
$(1+w_1+\dots+w_k)^{-1}=1+w_1+\dots+w_k+(w_1+\dots+w_k)^2+\dots$ in dimension $i$. Of course,
we have
\begin{equation}\label{recurrentformforbarw_i} {\bar w}_{i}= w_1{\bar w}_{i-1}+w_2{\bar
w}_{i-2}+\cdots +w_{k}{\bar w}_{i-k}. \end{equation} We note that ${\bar w}_i$ represents the
$i$th dual Stiefel-Whitney class of $\gamma$, that is, the Stiefel-Whitney class
$w_i(\gamma^\perp_{n,k})\in H^i(G_{n,k})$ of the complementary $(n-k)$-plane bundle
$\gamma^\perp_{n,k}$ (briefly $\gamma^\perp$); we shall also use ${\bar w}_i$ as an
abbreviation for $w_i(\gamma^\perp)$.

By what we have said, no nonzero homogeneous polynomials in $w_1,\dots, w_k$ in dimensions
$\leq n-k$ represent $0$ in cohomology; therefore the kernel \eqref{kernelofmultipbyw_1} is
the zero-subgroup for all $j\leq n-k-1$, and we always have
\begin{equation}\mathrm{charrank}(\widetilde \gamma_{n,k})\geq n-k-1.
\end{equation}

For the Grassmann manifold $G_{n,k}$ $(3\leq k \leq n-k)$, let $g_i(w_2, \dots, w_k)$
$($briefly just $g_i$$)$ denote the reduction of ${\bar w}_i(w_1, \dots, w_k)$ modulo $w_1$.

The following fact is obvious.

\begin{fact}\label{fact:fromhighertolower}
Let $r<k$. If ${\bar w}_i(w_1, \dots, w_k)=0$, then also ${\bar w}_i(w_1, \dots, w_r)=0$ and,
similarly, if $g_i(w_2, \dots, w_k)=0$, then also $g_i(w_2, \dots, w_r)=0$.
\end{fact}

For $G_{n,k}$, the formula \eqref{recurrentformforbarw_i} implies that $g_i=w_2 g_{i-2} + w_3
g_{i-3}+ \cdots +w_k g_{i-k}$, and an obvious induction proves that

\begin{equation} \label{recurformforg_iforG_{n,k}}
g_i=w_2^{2^s} g_{i-2\cdot 2^{s}} + w_3^{2^s} g_{i-3\cdot 2^s}+ \cdots + w_k^{2^s} g_{i-k\cdot
2^s}
\end{equation}
for all $s$ such that $i\geq 1+k\cdot 2^s$.

In our proof of Theorem \ref{th:charrankoforientedcanonicalbundles}, we shall use the
following.

\begin{lemma}\label{lemma:relationsmodw_1} For the Grassmann manifold
$G_{n,k}$ $(3\leq k \leq n-k)$,
\begin{enumerate}

\item[(i)] $g_i(w_2, w_3)=0$ if and only if $i=2^t - 3$ for some $t\geq 2$;

\item[(ii)] $g_i(w_2, w_3, w_4)=0$ if and only if $i=2^t - 3$ for some $t\geq 2$;

\item[(iii)] if $k\geq 5$ then, for $i\geq 2$, we never have $g_i(w_2, \dots, w_k)=0$.
\end{enumerate}

\end{lemma}

\begin{proof}[\rm Proof of Lemma \ref{lemma:relationsmodw_1}]

\underline{Part (i)}. In view of Fact \ref{fact:fromhighertolower}, the equality
\[g_{2^t-3}(w_2, w_3)=0\] for $t\geq 2$ (already proved, in a different way, in \cite{korbas2010})
is a direct consequence of the equality $g_{2^t-3}(w_2,
w_3, w_4)=0$; the latter will be verified in the proof of Part (ii).

Now we prove that $g_i(w_2, w_3)\neq 0$ for $i\neq 2^t-3$. For $i<14$, this is readily
verified by a direct calculation. Let us suppose that $i\geq 14$. Then, for each $i$, there
exists a uniquely determined integer $\lambda$ $(\lambda \geq 2)$ such that $ 2^\lambda < i/3
\leq 2^{\lambda +1}$. For proving the claim, it suffices to verify it in each of the
following three situations:

\begin{enumerate}
\item[(a)] $3\cdot 2^\lambda +1 \leq i < 5\cdot 2^\lambda$;
\item[(b)] $i = 5\cdot 2^\lambda$;
\item[(c)] $5\cdot 2^\lambda +1\leq i \leq 6\cdot 2^\lambda$.
\end{enumerate}

Case $(a)$. By \eqref{recurformforg_iforG_{n,k}}, we have
$$g_i=w_2^{2^\lambda} g_{i-2\cdot 2^{\lambda}} +
w_3^{2^\lambda} g_{i-3\cdot 2^\lambda}.$$ By our assumption, $i$ is not of the form $2^j -
3$, and one sees that $i-2\cdot 2^{\lambda}$ or $i-3\cdot 2^\lambda$ is not of the form $2^j
- 3$. If just one of the numbers $i-2\cdot 2^{\lambda}$, $i-3\cdot 2^\lambda$ is not of the
form $2^j - 3$, then it suffices to apply the inductive hypothesis (and the proved fact that
$g_{2^t-3}=0$ for $t\geq 2$). If none of the numbers $i-2\cdot 2^{\lambda}$ and $i-3\cdot
2^\lambda$ have the form $2^j - 3$ then, by the inductive hypothesis, both $g_{i-2\cdot
2^{\lambda}}$ and $g_{i-3\cdot 2^\lambda}$ are nonzero and, as a consequence, also $g_i\neq
0$. Indeed, now a necessary condition for $g_i= 0$ is that $g_{i-2\cdot 2^\lambda}$ should
contain the term $w_3^{2^\lambda}$; but the latter implies that $i-2\cdot 2^\lambda \geq
3\cdot 2^\lambda$, thus $i\geq 5\cdot 2^\lambda$, which is not fulfilled.

Case $(b)$. One directly sees, from $(1+w_2+w_3)^{-1}= 1+w_2+w_3+(w_2+w_3)^2+\cdots$, that
$$g_{5\cdot 2^\lambda} = w_2^{5\cdot 2^{\lambda -1}} + \text{ different\,\, terms} \neq 0.$$

Case $(c)$. By a repeated use of \eqref{recurformforg_iforG_{n,k}}, we now have that
\begin{align}
g_i&=w_2^{2^\lambda} (w_2^{2^\lambda}g_{i-4\cdot 2^\lambda} + w_3^{2^\lambda}g_{i-5\cdot
2^\lambda})\cr &+ w_3^{2^\lambda} (w_2^{2^{\lambda -1}}g_{i-4\cdot 2^\lambda}
+w_3^{2^{\lambda -1}}g_{i-9\cdot 2^{\lambda -1}})\cr &= (w_2^{2^{\lambda +1}} +
w_2^{2^{\lambda -1}}w_3^{2^\lambda})g_{i-4\cdot 2^\lambda}\cr
&+w_2^{2^\lambda}w_3^{2^\lambda}g_{i-5\cdot 2^\lambda} + w_3^{3\cdot 2^{\lambda
-1}}g_{i-9\cdot 2^{\lambda -1}}.
\end{align}

If $i-4\cdot 2^\lambda$ is of the form $2^j - 3$, then one verifies that $i-5\cdot
2^{\lambda}$ or $i-9\cdot 2^{\lambda -1}$ is not of the form $2^j - 3$. If just one of the
numbers $i-5\cdot 2^{\lambda}$, $i-9\cdot 2^{\lambda -1}$ is not of the form $2^j - 3$, then
it suffices to apply the inductive hypothesis (and the proved fact that $g_{2^t-3}=0$ for
$t\geq 2$). If none of the numbers $i-5\cdot 2^{\lambda}$ and $i-9\cdot 2^{\lambda -1}$  have
the form $2^j - 3$ then, by the inductive hypothesis, both $g_{i-5\cdot 2^{\lambda}}$ and
$g_{i-9\cdot 2^{\lambda -1}}$ are nonzero and, as a consequence, also $g_i\neq 0$. Indeed,
now a necessary condition for $g_i= 0$ is that $g_{i-5\cdot 2^{\lambda}}$ should contain the
term $w_3^{2^{\lambda -1}}$; but the latter implies that $i-5\cdot 2^{\lambda} \geq 3\cdot
2^{\lambda -1}$, thus $i> 6\cdot 2^\lambda$, which is not fulfilled.

Finally, let us suppose that $i-4\cdot 2^\lambda$ is not of the form $2^j - 3$ (thus, by the
inductive hypothesis, $g_{i-4\cdot 2^\lambda}\neq 0$). Then, in order to have $g_i=0$, it
would be necessary to \lq\lq eliminate\rq\rq\, $w_2^{2^{\lambda +1}}g_{i-4\cdot 2^\lambda}$.
This would only be possible if $g_{i-5\cdot 2^{\lambda}}$ contains $w_2^{2^\lambda}$, thus if
$i-5\cdot 2^{\lambda}\geq 2\cdot 2^{\lambda}$, hence $i\geq 7\cdot {2^\lambda}$, which is not
fulfilled, or if $g_{i-9\cdot 2^{\lambda -1}}$ contains $w_2^{2^{\lambda +1}}$, thus if
$i-9\cdot 2^{\lambda -1} \geq 2\cdot 2^{\lambda +1}$, hence $i\geq 17\cdot 2^{\lambda -1}
\geq 8\cdot 2^\lambda$, which is not fulfilled.

\underline{Part (ii)}. We first prove that $g_{2^t - 3}(w_2, w_3, w_4)=0$ for $t\geq 2$. We
directly see that $g_1=0$ and $g_5=0$. For $t\geq 3$ we have, by
\eqref{recurformforg_iforG_{n,k}} and the inductive hypothesis, that
\begin{equation}
g_{2^t-3} =w_2^{2^{t-3}}g_{3\cdot 2^{t-2}-3} + w_3^{2^{t-3}}g_{5\cdot 2^{t-3}-3}.
\end{equation}

Thus, again by \eqref{recurformforg_iforG_{n,k}} and the inductive hypothesis, we obtain
\begin{align} g_{2^t-3} &= w_2^{2^{t-3}}(w_2^{2^{t-3}}g_{2^{t-1}-3}+w_3^{2^{t-3}}
g_{3\cdot 2^{t-3} -3}+w_4^{2^{t-3}}g_{2^{t-2} -3})\cr &+ w_3^{2^{t-3}}(w_2^{2^{t-3}}g_{3\cdot
2^{t-3} -3}+w_3^{2^{t-3}}g_{2^{t-2} -3}+w_4^{2^{t-3}}g_{2^{t-3}-3})\cr &=0.
\end{align}

\underline{Part (iii)}. First, one readily calculates that $g_{5}(w_2, w_3, w_4, w_5)=w_5
\neq 0$. Then for completing the proof of Part (iii), in view of what we have proved up to
now and Fact \ref{fact:fromhighertolower}, it suffices to verify that $g_{2^t -3}(w_2, w_3,
w_4, w_5)\neq 0$ for $t\geq 4$. For this, we show that $h_{2^t -3}(w_4, w_5)$ is nonzero for
$t\geq 4$, where $h_{2^t -3}(w_4, w_5)$ (briefly $h_{2^t -3}$) is obtained by reducing
$g_{2^t -3}(w_2, w_3, w_4, w_5)$ modulo $w_2$ and $w_3$. Indeed, by
\eqref{recurformforg_iforG_{n,k}}, we see that
\begin{equation}
h_{2^t -3} = w_4^{2^{t-3}}h_{2^{t-1} -3} + w_5^{2^{t-3}}h_{3\cdot 2^{t-3} -3}.
\end{equation}
By the inductive hypothesis, $h_{2^{t-1} -3}\neq 0$; thus a necessary condition for $h_{2^t
-3}=0$ is that the term $w_5^{2^{t-3}}$ should be contained in $h_{2^{t-1} -3}$. But this
would require that $2^{t-1} -3\geq 5\cdot {2^{t-3}}$, which is not fulfilled. This finishes
the proof of Lemma \ref{lemma:relationsmodw_1}.
\end{proof}

The announced preparations are finished, and we can prove Theorem
\ref{th:charrankoforientedcanonicalbundles}.

\begin{proof}[\rm Proof of Theorem
\ref{th:charrankoforientedcanonicalbundles}] Recall that, for $G_{n,k}$ $(k\leq n-k)$ there
are no polynomial relations among $w_1, w_2, \dots, w_k$ in dimensions $\leq n-k$, and a
nonzero polynomial $p_{n-k+1}\in \mathbb Z_2[w_1, w_2, \dots, w_k]$ represents $0\in
H^{n-k+1}(G_{n,k})$ if and only if $p_{n-k+1}={\bar w}_{n-k+1}$. From the Gysin sequence
\eqref{Gysinsequence} we see that
\begin{align}\label{criterionfordegree(n-k)}
&p^*: H^{n-k}(G_{n,k}) \longrightarrow H^{n-k}(\widetilde G_{n,k}) \text{ is\,
surjective\,}\cr &\text{and,\, equivalently,\,\,} \mathrm{charrank}(\widetilde
\gamma_{n,k})\geq n-k,\cr &\text{precisely\, when\,} g_{n-k+1}(w_2, \dots, w_k)\neq 0.
\end{align}
We still observe that, for $3\leq k \leq n-k$,
\begin{equation}\label{whenfirsttwogenrelatsarenonzeromodw_1}
\text{ if } g_{n-k+1}\neq 0 \text{ and } g_{n-k+2}\neq 0, \text{ then }
\mathrm{charrank}(\widetilde \gamma_{n,k})\geq n-k+1. \end{equation}

Indeed, by the criterion \eqref{criterionfordegree(n-k)}, we have
$\mathrm{charrank}(\widetilde \gamma_{n,k})\geq n-k$. To show that this inequality can be
improved as claimed in \eqref{whenfirsttwogenrelatsarenonzeromodw_1}, let us suppose that a
nonzero polynomial $p_{n-k+1}\in \mathbb Z_2[w_1, \dots, w_k]$ represents an element in
$\mathrm{Ker}(H^{n-k+1}(G_{n,k})\overset{w_1}{\longrightarrow} H^{n-k+2}(G_{n,k}))$. Thus
$w_1p_{n-k+1}$ represents $0\in H^{n-k+2}(G_{n,k})$. This means that, in $\mathbb Z_2[w_1,
\dots, w_k]$, $w_1p_{n-k+1}= aw_1{\bar w}_{n-k+1}+b{\bar w}_{n-k+2}$, where $a=1$ or $b=1$.
Of course, since $g_{n-k+2}\neq 0$, necessarily $b=0$, $a=1$. But the polynomial equality
$w_1p_{n-k+1}= w_1{\bar w}_{n-k+1}$ implies that $p_{n-k+1}= {\bar w}_{n-k+1}$, thus
$p_{n-k+1}$ represents $0\in H^{n-k+1}(G_{n,k})$. So we see that
$\mathrm{Ker}(H^{n-k+1}(G_{n,k})\overset{w_1}{\longrightarrow} H^{n-k+2}(G_{n,k}))=0$ and
$\mathrm{charrank}(\widetilde \gamma_{n,k})\geq n-k+1$.

\underline{Proof of Parts (1) and (2)}. By Lemma \ref{lemma:relationsmodw_1}(i), (ii),
$g_{n-k+1}(w_2, \dots, w_k)$ vanishes if $(n,k)\in \{(2^t-1,3), (2^t,4)\}$. By the criterion
\eqref{criterionfordegree(n-k)}, for these pairs $(n,k)$, the homomorphism $p^*:
H^{n-k}(G_{n,k}) \longrightarrow H^{n-k}(\widetilde G_{n,k})$ is not surjective; thus, there
is a non-Stiefel-Whitney generator in $H^{n-k}(\widetilde G_{n,k})$ if $(n,k)\in \{(2^t-1,3),
(2^t,4)\}$, and we conclude that $\mathrm{charrank}(\widetilde \gamma_{2^t-1,3})=2^t-5 =
\mathrm{charrank}(\widetilde \gamma_{2^t,4})$.

Of course, again by Lemma \ref{lemma:relationsmodw_1}(i), (ii), we have $g_{n-k+1}(w_2,
\dots, w_k)\neq 0$ if $(n,k)\not\in \{(2^t-1,3), (2^t,4)\}$ and $k\in \{3,\, 4\}$. By the
criterion \eqref{criterionfordegree(n-k)}, for these pairs $(n,k)$, the homomorphism $p^*:
H^{n-k}(G_{n,k}) \longrightarrow H^{n-k}(\widetilde G_{n,k})$ is surjective; so we have that
$\mathrm{charrank}(\widetilde \gamma_{n,3})\geq n-3$ if $n\neq 2^t-1$ and
$\mathrm{charrank}(\widetilde \gamma_{n,4})\geq n-4$ if $n\neq 2^t$.

To prove the result for $\widetilde G_{2^t -2,3}$, we first recall (Lemma
\ref{lemma:relationsmodw_1}(i)) that $g_{2^t-4}\neq 0,\, g_{2^t-3}=0$, and $g_{2^t-2}\neq 0$.
Thus ${\bar w}_{2^t-3}=w_1 p_{2^t-4}$ for some polynomial $p_{2^t-4}$. The latter cannot
represent $0$ in the cohomology group $H^{2^t-4}(G_{2^t -2,3})$; indeed, if $p_{2^t-4}$
represents zero, then necessarily $p_{2^t-4}={\bar w}_{2^t-4}$ (as polynomials), thus we have
a relation ${\bar w}_{2^t-3}=w_1 {\bar w}_{2^t-4}$, which is impossible. This implies (see
\eqref{Gysinsequence}) that $p^\ast: H^{2^t-4}(G_{2^t -2,3}) \rightarrow H^{2^t-4}(\widetilde
G_{2^t -2,3})$ is not an epimorphism, thus $\mathrm{charrank}(\widetilde
\gamma_{2^t-2,3})\leq 2^t-5$. By \eqref{criterionfordegree(n-k)}, since $g_{2^t-4}\neq 0$, we
have $\mathrm{charrank}(\widetilde\gamma_{2^t-2,3})\geq 2^t-5$, which proves the claim for
$\widetilde G_{2^t -2,3}$. The result for $\widetilde G_{2^t -1,4}$ can be derived in an
analogous way.

Now we prove the claim for $\widetilde G_{2^t -3,3}$. We have $g_{2^t-5}\neq 0,\,
g_{2^t-4}\neq 0$, and $g_{2^t-3}= 0$. Thus ${\bar w}_{2^t-3}=w_1 p_{2^t-4}$ for some
polynomial $p_{2^t-4}$. The latter cannot represent $0$ in $H^{2^t-4}(G_{2^t -3,3})$. Indeed,
if $p_{2^t-4}$ represents zero, then $p_{2^t-4}=aw_1{\bar w}_{2^t-5}+b{\bar w}_{2^t-4}$ in
$\mathbb Z_2[w_1, w_2, w_3]$, with $a=1$ or $b=1$; as a consequence, we would have ${\bar
w}_{2^t-3}=aw_1^2{\bar w}_{2^t-5}+bw_1{\bar w}_{2^t-4}$, which is impossible. From the Gysin
sequence \eqref{Gysinsequence}, we see that $p^\ast: H^{2^t-4}(G_{2^t -3,3}) \rightarrow
H^{2^t-4}(\widetilde G_{2^t -3,3})$ is not an epimorphism. Thus $\mathrm{charrank}(\widetilde
\gamma_{2^t-3,3})\leq 2^t-5$. At the same time, by the observation
\eqref{whenfirsttwogenrelatsarenonzeromodw_1}, we have $\mathrm{charrank}(\widetilde
\gamma_{2^t-3,3})\geq 2^t-5$. This proves the claim for $\widetilde G_{2^t -3,3}$; again, the
result for $\widetilde G_{2^t -2,4}$ can be proved analogously.

We pass to proving the result for $\widetilde G_{2^t,3}$. We know that none of $g_{2^t -2},
\,g_{2^t -1}, \, g_{2^t}$ vanishes. By \eqref{whenfirsttwogenrelatsarenonzeromodw_1}, we see
that $\mathrm{charrank}(\widetilde \gamma_{2^t,3})\geq 2^t-2$. At the same time, since $w_2
g_{2^t-2} + g_{2^t}=w_3 g_{2^t-3}=0$, we have (as for $\mathbb Z_2$-polynomials) $w_2 {\bar
w}_{2^t-2} + {\bar w}_{2^t}=w_1 p_{2^t-1}$, for some polynomial $p_{2^t-1}$. The latter
cannot represent $0\in H^{2^t-1}(G_{2^t,3})$. Indeed, $p_{2^t-1}$ representing $0$ would mean
that $p_{2^t-1}=a w_1{\bar w}_{2^t-2}+b{\bar w}_{2^t-1}$ (where $a=1$ or $b=1$), which
implies an impossible relation ${\bar w}_{2^t}=(aw_1^2 +w_2){\bar w}_{2^t-2}+b w_1{\bar
w}_{2^t-1}$. Thus $p_{2^t-1}$ represents a nonzero element in
\[\mathrm{Ker}(H^{2^t-1}(G_{2^t,3})\overset{w_1}{\longrightarrow} H^{2^t}(G_{2^t,3})),\] and we
have that $\mathrm{charrank}(\widetilde \gamma_{2^t,3})\leq 2^t-2$, which proves the claim
for $\widetilde G_{2^t,3}$.

Now we shall pass to $\widetilde G_{2^t-3,4}$. Then we have $g_{2^t -6}\neq 0, \,g_{2^t
-5}\neq 0, \, g_{2^t-4}\neq 0,\,g_{2^t-3} = 0$. By
\eqref{whenfirsttwogenrelatsarenonzeromodw_1}, we know that $\mathrm{charrank}(\widetilde
\gamma_{2^t-3,4})\geq 2^t-6$. To improve this inequality, we now show that
\begin{equation}\label{ker(fromH^{2^t-5}(G_{2^t-3,4})toH^{2^t-4}(G_{2^t-3,4}))=zero}
\mathrm{Ker}(H^{2^t-5}(G_{2^t-3,4})\overset{w_1}{\longrightarrow}
H^{2^t-4}(G_{2^t-3,4}))=0.\end{equation} Let a nonzero polynomial $p_{2^t-5}$ represent an
element in the kernel under question. This means that the polynomial $w_1p_{2^t-5}$
represents $0\in H^{2^t-4}(G_{2^t-3,4})$. Consequently, $w_1p_{2^t-5}= aw_1^2{\bar
w}_{2^t-6}+bw_2{\bar w}_{2^t-6}+cw_1{\bar w}_{2^t-5}+d{\bar w}_{2^t-4}$ in $\mathbb Z_2[w_1,
w_2, w_3, w_4]$, where at least one of the coefficients $a,\,b,\,c,\,d$ is equal to $1$. We
cannot have $b=d=1$, because $w_2{\bar w}_{2^t-6}+{\bar w}_{2^t-4}$ reduced mod $w_1$ is
$w_2{g}_{2^t-6}+{g}_{2^t-4}$ and, as we shall see in the next step, the latter is not zero.
Indeed, let $z_i$ denote the reduction of $g_i$ modulo $w_2$ and $w_3$. Then $w_2
g_{2^t-6}+g_{2^t-4}$ reduced modulo $w_2$ and $w_3$ is equal to $z_{2^t-4}$. A direct
calculation gives that $z_{12}=w_4^3$ and, by induction, we obtain that
$z_{2^t-4}=w_4^{2^{t-3}}z_{2^{t-1}-4}=w_4^{2^{t-3}}w_4^{2^{t-3}-1}=w_4^{2^{t-2}-1}\neq 0$. So
we have shown that $w_2 g_{2^t-6}+g_{2^t-4}\neq 0$. One also readily sees that it is
impossible to have $(b,d)=(1,0)$ as well as $(b,d)=(0,1)$. Thus the only remaining
possibility is $(b,d)=(0,0)$. So we obtain $w_1p_{2^t-5}= w_1 (aw_1{\bar w}_{2^t-6}+{c\bar
w}_{2^t-5})$, thus $p_{2^t-5}= aw_1{\bar w}_{2^t-6}+c{\bar w}_{2^t-5}$. This means that
$p_{2^t-5}$ represents $0\in H^{2^t-5}(G_{2^t-3,4})$, and we have proved the equality
\eqref{ker(fromH^{2^t-5}(G_{2^t-3,4})toH^{2^t-4}(G_{2^t-3,4}))=zero}.

As a consequence, we have $\mathrm{charrank}(\widetilde \gamma_{2^t-3,4})\geq 2^t-5$. Since
$g_{2^t-3}=0$, we have that ${\bar w}_{2^t-3}=w_1 p_{2^t-4}$ for some polynomial $p_{2^t-4}$,
about which one can show (similarly to situations of this type dealt with above) that it
cannot represent zero in cohomology. Thus we also have $\mathrm{charrank}(\widetilde
\gamma_{2^t-3,4})\leq 2^t-5$, and finally $\mathrm{charrank}(\widetilde \gamma_{2^t-3,4})=
2^t-5$.

In view of Lemma \ref{lemma:relationsmodw_1}(i), (ii), for all the manifolds $\widetilde
G_{n,3}$ and $\widetilde G_{n,4}$ that remain, the observation
\eqref{whenfirsttwogenrelatsarenonzeromodw_1} implies the lower bounds stated in Theorem
\ref{th:charrankoforientedcanonicalbundles}(1),(2).

\underline{Proof of Part (3)}. For $k\geq 5$, Lemma \ref{lemma:relationsmodw_1}(iii) says
that $g_{n-k+1}\neq 0$ and $g_{n-k+2}\neq 0$; thus the observation
\eqref{whenfirsttwogenrelatsarenonzeromodw_1} applies, giving that
$\mathrm{charrank}(\widetilde \gamma_{n,k})\geq n-k+1$ in all these cases.

To prove the final statement of the theorem, it suffices to recall that, if $n$ is odd, then
(see \cite[p. 72]{korbas2010}) we have $w_i(\widetilde G_{n,k})= \widetilde w_i +
Q_i(\widetilde w_2, \dots, \widetilde w_{i-1})$ $(i\leq k)$, where $Q_i$ is a $\mathbb
Z_2$-polynomial, and $\widetilde w_j = w_j(\widetilde G_{n,k}) + P_j(w_2(\widetilde G_{n,k}),
\dots, w_{j-1}(\widetilde G_{n,k}))$ $(j\geq 2)$ for some $\mathbb Z_2$-polynomial $P_j$.

The proof of Theorem \ref{th:charrankoforientedcanonicalbundles} is finished.

\end{proof}

\section{On the cup-length of the Grassmann manifold $\widetilde G_{n,k}$}
\label{sec:newresultsoncup-length} Recall that the $\mathbb Z_2$-cup-length,
$\mathrm{cup}(X)$, of a compact path connected topological space $X$ is defined to be the
maximum of all numbers $c$ such that there exist, in positive degrees, cohomology classes
$a_1,\dots,a_c\in H^\ast(X)$ such that their cup product $a_1\cdots a_c$ is nonzero. In
\cite{korbas2010} and, independently, in \cite{fukaya}, it was proved that for $t\geq 3$ we
have
\[
\mathrm{cup}(\widetilde G_{2^t - 1,3})=2^t - 3;\] in addition, \cite[Theorem 1.3]{korbas2010}
gave certain upper bounds for $\mathrm{cup}(\widetilde G_{n,k})$.

Now Theorem \ref{th:charrankoforientedcanonicalbundles} implies the following exact result
for $\widetilde G_{2^t,3}$, confirming the corresponding claim in Fukaya's conjecture
\cite[Conjecture 1.2]{fukaya}, or improvements on the results of \cite[Theorem
1.3]{korbas2010} in the other cases.

\begin{theorem} \label{thm:cup-lengthoforGrassmmanifs} For the oriented Grassmann
manifold $\widetilde G_{n,k}$ $(3\leq k\leq n-k)$, with $2^{t-1}<n\leq 2^t$, we have

\begin{enumerate}

\item $\mathrm{cup}(\widetilde G_{n,3})
\left\{\begin{array}{ll}=n-3&\text{ if } n = 2^t,\\

\leq (2n-3-i)/2 &\text{ if } n=2^t -i,\,i\in \{2,\,3\},\\
\leq n-3 &\text{ otherwise, for $n\neq 2^t -1$; }
\end{array}\right.$
\vspace{4mm}

\item $\mathrm{cup}(\widetilde G_{n,4})
\left\{\begin{array}{ll}
\leq (3n-10-i)/2 &\text{ if } n = 2^t-i,\,i\in \{0,\,1,\,2,\,3\},\\

\leq (3n-12)/2 &\text{ otherwise; } \\
\end{array}\right.$
\vspace{4mm}
\item if $k\geq 5$, then $\mathrm{cup}(\widetilde G_{n,k}) \leq \dfrac{(k-1)(n-k)}{2}$.
\end{enumerate}

\end{theorem}

\begin{proof}[\rm Proof] For a connected finite $CW$-complex $X$,
let $r_X$ denote the smallest positive integer such that $\widetilde{H}^{r_X}(X)\neq 0$. In
the case that such an integer does not exist, that is, all the reduced cohomology groups
$\widetilde{H}^i(X)$ $(1\leq i\leq \mathrm{dim}(X))$ vanish, we set $r_X=\mathrm{dim}(X)+1$;
thus always $r_X\geq 1$. To obtain the upper bounds stated in the theorem, we use the
following generalization of \cite[Theorem 1.1]{korbas2010}.

\begin{theorem}[A. Naolekar - A. Thakur \cite{naolekarthakur}]\label{Naolekar-Thakuroncup-length}
Let $X$ be a connected closed smooth $d$-dimensional manifold. Let $\xi$ be a vector bundle
over $X$ satisfying the following: there exists $j$, $j\leq \mathrm{charrank}_X(\xi)$, such
that every monomial $w_{i_1}(\xi)\cdots w_{i_r}(\xi)$, $0\leq i_t\leq j$, in dimension $d$
vanishes. Then
$$\mathrm{cup}(X)\leq 1+\frac{d-j-1}{r_X}.$$
\end{theorem}

For the manifold $\widetilde G_{n,k}$, every top-dimensional monomial in the Stiefel-Whitney
classes of the canonical bundle $\widetilde \gamma_{n,k}$ vanishes (indeed, if a
top-dimensional monomial in the Stiefel-Whitney classes of  $\widetilde \gamma_{n,k}$ does
not vanish, then it is a $p^\ast$-image of the corresponding non-vanishing top-dimensional
monomial in the Stiefel-Whitney classes of  $\gamma_{n,k}$; due to Poincar\'e duality, the
latter monomial can be replaced with a monomial divisible by $w_1(\gamma_{n,k})$; but
$p^\ast$ maps this monomial to zero). Now the upper bounds stated in Theorem
\ref{thm:cup-lengthoforGrassmmanifs} are obtained by taking $X=\widetilde G_{n,k}$ $(3\leq
k\leq n-k)$, $\xi=\widetilde \gamma_{n,k}$, and $j$ equal to the right-hand side of the
corresponding (in)equality given in Theorem \ref{th:charrankoforientedcanonicalbundles}.

For $\widetilde G_{2^t,3}$, it was proved in \cite[p. 77]{korbas2010} that $w_2(\widetilde
\gamma)^{2^t-4}$ does not vanish. This implies that $\mathrm{cup}(\widetilde G_{2^t,3})\geq
2^t-3$; this lower bound coincides with the upper bound proved above. The proof is finished.
\end{proof}

\begin{ack} The author thanks Peter Zvengrowski and the referee for useful comments related to the
presentation of this paper.
\end{ack}

\end{document}